\documentclass[11pt]{amsart}
\usepackage{amsmath}
\usepackage{etex}
\usepackage{amsfonts}
\usepackage{amssymb}
\usepackage{mathrsfs}
\usepackage{amsthm}
\usepackage{cite}
\usepackage{verbatim}
\usepackage{pictex}
\usepackage{dcpic}
\usepackage[margin=3.5cm]{geometry}

\newtheorem{Theorem}{Theorem}[section]
\newtheorem{Corollary}[Theorem]{Corollary}
\newtheorem{Lemma}[Theorem]{Lemma}
\newtheorem{Proposition}[Theorem]{Proposition}

\newtheorem*{Convention*}{Convention}
\newtheorem*{BasicPrinciple*}{Basic Principle}

\theoremstyle{definition}
\newtheorem{Definition}{Definition}[section]

\theoremstyle{remark}
\newtheorem{Remark}{Remark}[section]

\newcommand{\FF}{\mathcal{F}}
\newcommand{\GG}{\mathcal{G}}

\newcommand{\LL}{\mathcal{L}}
\newcommand{\MM}{\mathcal{M}}

\newcommand{\OO}{\mathcal{O}}

\newcommand{\SSS}{\mathcal{S}}
\newcommand{\TT}{\mathcal{T}}

\newcommand{\VV}{\mathcal{V}}

\newcommand{\SL}{\textnormal{SL}}

\newcommand{\complex}{\mathbb{C}}
\newcommand{\Grass}{\textnormal{Gr}}
\newcommand{\Rep}{\textnormal{Rep}}
\newcommand{\Hom}{\textnormal{Hom}}
\newcommand{\Ext}{\textnormal{Ext}}
\newcommand{\SI}{\textnormal{SI}}
\newcommand{\Spec}{\textnormal{Spec}}

\begin{document}
\title{Quiver Generalization of a Conjecture of King, Tollu, and Toumazet}
\author{Cass Sherman}
\address{Department of Mathematics, University of North Carolina at Chapel Hill, Phillips Hall, Chapel Hill, NC 27599}
\email{cas1987@email.unc.edu}
\begin{abstract}
Stretching the parameters of a Littlewood-Richardson coefficient of value $2$ by a factor of $n$ results in a coefficient of value $n+1$ \cite{KTT}\cite{Ikenmeyer}\cite{ShermanKTT}. We give a geometric proof of a generalization for representations of quivers.
\end{abstract}
\maketitle

\begin{section}{Introduction}\label{Introduction}
The Littlewood-Richardson coefficients $c^{\nu}_{\lambda,\mu}$ arise in the representation theory of the general linear group. They depend on tuples of nonnegative integers (weights) $\lambda$, $\mu$, and $\nu$. An operation called stretching can be performed in which all of the integers in the tuples $\lambda$, $\mu$, and $\nu$ are multiplied by $n$. The effect of this on the Littlewood-Richardson coefficient, that is, the function $P(n):=c^{n\nu}_{n\lambda,n\mu}$, has been studied by many. A number of new and existing conjectures on the behavior of $P$ were summarized by King et al. \cite{KTT}. We list some of these below:
\begin{itemize}
\item (Polynomiality Conjecture) $P$ is a polynomial.
\item (Saturation Conjecture) If $P(1)=0$, then $P(n)=0$ for all $n\geq 1$.
\item (Fulton's Conjecture) If $P(1)=1$, then $P(n)=1$ for all $n\geq 1$.
\item (KTT Conjecture) If $P(1)=2$, then $P(n)=n+1$ for all $n\geq 1$.
\end{itemize}    
The polynomiality conjecture was proven by Derksen and Weyman \cite{DerksenWeymanPolynomiality}. The first (combinatorial) proofs of the saturation and Fulton conjectures are due to Knutson, Tao, and Woodward \cite{KT}, \cite{KTW2}. Subsequent geometric proofs appeared from Belkale \cite{Bel06}, \cite{Bel07} and others, which allow for an arbitrary number of weights after symmetrizing. The KTT conjecture was proven combinatorially by Ikenmeyer \cite{Ikenmeyer} for three weights, and geometrically by the author \cite{ShermanKTT}, again symmetrizing and allowing for an arbitrary number of weights.

For $\alpha,\beta$ dimension vectors of a cycle-free quiver $Q$ with Ringel product $0$, the dimensions of the spaces of $\sigma_{\beta}$-semi-invariant functions $\SI(Q,\alpha)_{\sigma_{\beta}}$ on $\Rep(Q,\alpha)$ appear to exhibit the same behavior under stretching as the Littlewood-Richardson numbers (see Section \ref{QuiverPrelims} for notation and generalities on quiver representations). Thus, one can make the same assertions for the function $\widetilde{P}(n):=\dim\SI(Q,\alpha)_{\sigma_{n\beta}}$.
\begin{itemize}
\item (Polynomiality) $\widetilde{P}$ is a polynomial.
\item (Saturation) If $\widetilde{P}(1)=0$, then $\widetilde{P}(n)=0$ for all $n\geq 1$. 
\item (Fulton) If $\widetilde{P}(1)=1$, then $\widetilde{P}(n)=1$ for all $n\geq 1$.
\end{itemize}
All of the above were proven by Derksen and Weyman in the papers \cite{DerksenWeymanPolynomiality}, \cite{DerksenWeymanSemiInvariants}, and \cite{DerksenWeymanCombinatorics}, respectively, with input from Belkale on the last of these. It is well-known that the results for $\widetilde{P}$ imply those for $P$, the Littlewood-Richardson numbers coinciding with dimensions of spaces of semi-invariant functions for special choices of $Q$, $\alpha$, $\beta$ (see Section \ref{LWConnection} for one approach). The main object of this paper is to establish the corresponding quiver generalization of the KTT Conjecture. That is, we prove: 

\begin{Theorem}\label{QuiverKTT}
Let $\alpha$, $\beta$ be dimension vectors of $Q$, a quiver without oriented cycles, such that $\langle\alpha,\beta\rangle_Q=0$. If $\dim\SI(Q,\alpha)_{\sigma_\beta}=2$, then $\dim\SI(Q,\alpha)_{\sigma_{n\beta}}=n+1$ for all positive integers $n$.
\end{Theorem}  

Our approach proceeds through geometric invariant theory, following similar proofs in \cite{Bel07}, \cite{ShermanKTT}. Along the way, we prove by dimension counting a result of general interest, Proposition \ref{QuiverH1PropStatement}. It has the flavor of results from Schofield's paper \cite{Schofield}, in that it equates $\Ext_{Q}(V,W)$ with $\Ext_{Q}(S,W)$, where $S$ is a certain subrepresentation of $V$. 

In the last section, we show how to deduce the main result of the author's paper \cite{ShermanKTT} (restated as Corollary \ref{QuiverInvariantsGiveTensorInvariants} here) from Theorem \ref{QuiverKTT}. 
\end{section}

\begin{section}{Preliminaries and Notation on Quiver Representations}\label{QuiverPrelims}

A \textit{quiver} $Q$ consists of the data of a pair finite sets $Q_0$ and $Q_1$ of vertices and arrows between vertices, respectively, along with maps $h,t:Q_1\rightarrow Q_0$, where the head map $h$ associates to each arrow the vertex of its pointer, and the tail map $t$ associates to each arrow the vertex of its base. We will assume moreover that a quiver has no oriented cycles when regarded as a digraph. 

A \textit{dimension vector} $\alpha$ is a function $\alpha:Q_0\rightarrow\mathbb{N}\cup\{0\}$. A \textit{representation} of $Q$ of dimension vector $\alpha$ is defined to be an element $V$ of the set:
\begin{equation*}
\Rep(Q,\alpha):=\prod_{a\in Q_1}\Hom(\complex^{\alpha(ta)},\complex^{\alpha(ha)})
\end{equation*}
We will frequently regard $\Rep(Q,\alpha)$ as an affine variety, by the obvious identification with $\mathbb{A}^N$ for $N=\sum_a\alpha(ta)\alpha(ha)$.

If $V$ and $W$ are representations of $Q$ of dimension vectors $\alpha$ and $\beta$, respectively, then a morphism $\phi:V\rightarrow W$ of quiver representations is, for each $x\in Q_0$, a homomorphism of vector spaces $\phi(x):\complex^{\alpha(x)}\rightarrow\complex^{\beta(x)}$, where these must satisfy the commutativity property $\phi(ha)\circ V(a)=W(a)\circ\phi(ta)$ for every $a$ in $Q_1$. The vector space $\Hom_Q(V,W)$ of all morphisms of quiver representation is then the kernel of the map
\begin{equation*}
d^V_W=\oplus_{x\in Q_0}\Hom(\complex^{\alpha(x)},\complex^{\beta(x)})\rightarrow\oplus_{a\in Q_1}\Hom(\complex^{\alpha(ta)},\complex^{\beta(ha)})
\end{equation*}
which sends  $\{\phi(x)\}_{x\in Q_0}$ to the element $\{W(a)\circ\phi(ta)-\phi(ha)\circ V(a)\}_{a\in Q_1}$. 

Let $\Rep(Q)$ denote the category with representations of $Q$ (of any dimension vector) as objects and the above notion of morphism. It is an abelian category. For representations $V$ and $W$, one has $\Ext^1(V,W)=\textnormal{coker}(d^V_W)$, and there is no higher $\Ext$ in this category, so we simply denote this cokernel by $\Ext_Q(V,W)$.

The (in general, nonsymmetric) \textit{Ringel form} on the abelian group of functions $Q_0\rightarrow \mathbb{Z}$ is the bilinear form:
\begin{equation}\label{RingelForm}
\langle\alpha,\beta\rangle_Q=\sum_{x\in Q_0}\alpha(x)\beta(x)-\sum_{a\in Q_1}\alpha(ta)\beta(ha)
\end{equation}
It is clear that if moreover $\alpha$ and $\beta$ are dimension vectors, then $\langle\alpha,\beta\rangle_Q$ is the difference of the dimensions of the domain and codomain of $d^V_W$, whence
\begin{equation}\label{RingelFormAlt}
\langle\alpha,\beta\rangle_Q=\dim\Hom_Q(V,W)-\dim\Ext_Q(V,W)
\end{equation}
for any representations $V,W$ of dimensions $\alpha,\beta$. In particular, the right hand side of (\ref{RingelFormAlt}) does not depend on $V$ and $W$ beyond their dimension vectors.

The affine variety $\Rep(Q,\alpha)$ has a natural action of 
\begin{equation*}
\textnormal{GL}(Q,\alpha):=\prod_{x\in Q_0}\textnormal{Aut}(\complex^{\alpha(x)})
\end{equation*}
given by conjugation: $g=(A(x))_{x}$ sends $V=(V(a))_a$ to $gV=( A(ha)V(a)A(ta)^{-1})_a$. Let the subgroup $\SL(Q,\alpha)$ of $\textnormal{GL}(Q,\alpha)$ be the product of the determinant $1$ subgroups of each factor in the product defining $\textnormal{GL}(Q,\alpha)$. We are interested in the rings of semi-invariants
\begin{equation*}
\SI(Q,\alpha)=(H^0(\Rep(Q,\alpha),\OO))^{\SL(Q,\alpha)},
\end{equation*}
where $\OO$ is the structure sheaf. These decompose into direct sums of weight spaces, called spaces of \emph{$\sigma$ semi-invariants}:
\begin{equation*}
\SI(Q,\alpha)_{\sigma}=\{f\in H^0(\Rep(Q,\alpha),\OO):g\cdot f=\sigma(g)f\textnormal{ for all }g\in\textnormal{GL}(Q,\alpha)\},
\end{equation*} 
for $\sigma$ a multiplicative character of $\textnormal{GL}(Q,\alpha)$. Such a character must be a product over $Q_0$ of integral powers of the determinant characters on each factor of $\textnormal{GL}(Q,\alpha)$. A character $\sigma$ may therefore be identified with a function or weight (also called $\sigma$) $Q_0\rightarrow\mathbb{Z}$. Each such $\sigma$ defines a notion of semistability on $\Rep(Q,\alpha)$. 

\begin{Definition}\label{SigmaSemistability}
Given two weights $\sigma,\gamma:Q_0\rightarrow \mathbb{Z}$, one defines the evaluation of $\sigma$ at $\gamma$ to be
\begin{equation*}
\sigma(\gamma)=\sum_{x\in Q_0}\sigma(x)\gamma(x). 
\end{equation*}
A representation $V$ of $Q$ which satisfies $\sigma(\dim V)=0$ is said to be \emph{$\sigma$-semistable} if for every nonzero subrepresentation $S$ of $V$, one has $\sigma(\dim S)\leq 0$. The representation $V$ is \emph{$\sigma$-stable} if the inequality is always strict. 
\end{Definition}

To complete the notation for Theorem \ref{QuiverKTT}, we introduce the following definition.

\begin{Definition}\label{SigmaBeta}
For a dimension vector $\beta$ of $Q$, we define the function $\sigma_{\beta}:Q_0\rightarrow\mathbb{Z}$ by $\sigma_{\beta}(x)=-\beta(x)+\sum_{a:ta=x}\beta(ha)$. 
\end{Definition}
\begin{Remark}\label{SigmaBetaRemark}
Clearly $\sigma_{n\beta}=n\sigma_{\beta}$ for any positive integer $n$. Notice also that if $\alpha:Q_0\rightarrow\mathbb{Z}$ is a function, one has $\sigma_{\beta}(\alpha)=-\langle\alpha,\beta\rangle$.
\end{Remark}
\end{section}

\begin{section}{Translation via GIT}\label{QuiverGIT}
The goal of this section is to prove Proposition \ref{GITForQuivers}, which in turn gives Proposition \ref{QuiverTranslation}, the latter translating the main theorem \ref{QuiverKTT} into a form more adaptable to our geometric approach. Parts of \ref{GITForQuivers} are known from the literature (and are credited suitably below), but the author could not find a reference for the descent of the line bundle $L_{\sigma_{\beta}}$, hence its full proof here.

\begin{Proposition}\label{GITForQuivers}
Let $\alpha,\beta$ be dimension vectors for a quiver $Q$ without oriented cycles, such that $\langle \alpha,\beta\rangle_Q=0$. Let $\sigma_{\beta}:Q_0\rightarrow \mathbb{Z}$ be the associated weight (Definition \ref{SigmaBeta}). If $R^{SS}$ denotes the open set of $\sigma_{\beta}$-semistable points of $\Rep(Q,\alpha)$, then there is a good quotient $\pi:R^{SS}\rightarrow Y_{\alpha,\beta}$, where $Y_{\alpha,\beta}$ is an integral, projective $\complex$-variety with rational singularities (in particular, is normal). If $L_{\sigma_{\beta}}$ denotes the trivial line bundle $\Rep(Q,\alpha)\times\complex$ with $\textnormal{GL}(Q,\alpha)$-equivariant structure provided by $g\cdot(V,z)=(g\cdot V,\sigma_{\beta}(g^{-1})v)$ (now viewing $\sigma_{\beta}$ as a character, as in section \ref{QuiverPrelims}), then there exists an ample line bundle $L_{Y}$ on $Y_{\alpha,\beta}$ such that $\pi^*L_{Y}=L_{\sigma_{\beta}}|_{R^{SS}}$. Moreover, one has a canonical isomorphism $H^0(Y_{\alpha,\beta},L_Y^{\otimes n})= \SI(Q,\alpha)_{\sigma_{n\beta}}$. It follows from the saturation theorem of Derksen and Weyman \cite{DerksenWeymanSemiInvariants} that $Y_{\alpha,\beta}=\emptyset$ if and only if $\SI(Q,\alpha)_{\sigma_{\beta}}=0$.
\end{Proposition}

The next proposition paves the way for our geometric proof of Theorem \ref{QuiverKTT}. It follows from \ref{GITForQuivers} by a simple argument which appears in \cite[Theorem 2.5]{ShermanKTT}.

\begin{Proposition}\label{QuiverTranslation}
Theorem \ref{QuiverKTT} is equivalent to the following statement. If $\SI(Q,\alpha)_{\sigma_{\beta}}$ has dimension $2$, then $Y_{\alpha,\beta}$ has dimension $1$. 
\end{Proposition}

To prove \ref{GITForQuivers}, we begin with some generalities for a reductive group $G$ acting on the left on an affine $\complex$-variety $V=\Spec A$. Let $\sigma:G\rightarrow\complex^*$ be a character. Define a linearization $L$ of the action of $G$ on $V$ by letting the underlying bundle of $L$ be $V\times\complex$ and defining the action on $L$ such that $g\cdot(v,z)=(gv,\sigma(g^{-1})z)$. Writing $L^{-1}$ (also a trivial bundle) as $\Spec A[x]$, one obtains from the induced action a rational representation of $G$ on $A[x]$ by $g\cdot (fx^n)=(\sigma(g^{-n}))(f\circ g^{-1})x^n$; here we regard $f\in A$ as an algebraic function on $V$. Also, one has an action on global sections $s:V\rightarrow L$ by $g\cdot s=g\circ s\circ g^{-1}$. This gives rise to a grade-preserving action on $R=\oplus_{n=0}^{\infty}H^0(V,L^{\otimes n})$.

\begin{Lemma}\label{RIsAOfx}
With the above actions, one has a $G$-equivariant, graded $A$-algebra isomorphism $A[x]\rightarrow R$ given by sending $x$ to the constant section $1$ of $L$ in $R_1$. 
\end{Lemma}
\begin{proof}
The polynomial $fx^n$ goes to the section $s:v\mapsto (v,f(v))$ in $R_n$. The polynomial $g\cdot (fx^n)$ goes to to the section $v\mapsto(v,\sigma(g^{-n})(f\circ g^{-1})(v))$, which is $g\cdot s$.
\end{proof}
\begin{Remark}\label{RIsFiniteType}
Since $G$ acts rationally on $A[x]$, by the theorem of Hilbert/Nagata, $R^G$ is a finitely generated $\complex$ algebra.
\end{Remark}

Now, let $Q$ be a quiver without oriented cycles, and fix dimension vectors $\alpha,\beta$ with $\langle\alpha,\beta\rangle_Q=0$, and suppose $\sigma_{\beta}$ is as in Definition \ref{SigmaBeta}. Define a $\textnormal{GL}(Q,\alpha)$-equivariant line bundle $L_{\sigma_\beta}$ on $\Rep(Q,\alpha)$ as above.

\begin{Lemma}\label{SemiInvariantsAreSections}
For any $n\in\mathbb{N}$, we have $H^0(\Rep(Q,\alpha),L^{\otimes n}_{\sigma_\beta})^{\textnormal{GL}(Q,\alpha)}=\SI(Q,\alpha)_{\sigma_{n\beta}}$. 
\end{Lemma}
\begin{proof}
A section $f$ of $L^{\otimes n}_{\sigma_\beta}$ is simply a regular algebraic function on $\Rep(Q,\alpha)$. It is $\textnormal{GL}(Q,\alpha)$ invariant if and only if $f(gV)=(\sigma_{\beta}(g^{-1}))^{n}f(V)$ for all $V,g$. This happens if and only if $g^{-1}\cdot f=\sigma_{n\beta}(g^{-1})f$ for all $g$, i.e. if and only if $f$ is a $\sigma_{n\beta}$ semi-invariant. 
\end{proof}
Let $R_{\alpha,\beta}:=\oplus_{n=0}^{\infty}H^0(\Rep(Q,\alpha),L^n_{\sigma_{\beta}})^{\text{GL}(Q,\alpha)}=\oplus_{n=0}^{\infty}\SI(Q,\alpha)_{\sigma_{n\beta}}$ be the homogeneous coordinate ring of $Y_{\alpha,\beta}:=\text{Proj}(R_{\alpha,\beta})$. Note that
\begin{itemize}
\item $(R_{\alpha,\beta})_0=\complex$ since $Q$ has no oriented cycles \cite[Exercise 1.5.1.28]{SchmittGITAndDecoratedBundles}.
\item $Y_{\alpha,\beta}$ is a finite dimensional projective scheme over $\Spec((R_{\alpha,\beta})_0)=\Spec\complex$ by Remark \ref{RIsFiniteType}.
\item $Y_{\alpha,\beta}$ is a good quotient of $\Rep(Q,\alpha)^{SS}_{L_{\sigma_{\beta}}}$ by $\text{GL}(Q,\alpha)$ \cite{KingFDAlgebras}. Thus, $Y_{\alpha,\beta}$ is integral with rational singularities (in particular, is normal).
\item The notion of $L_{\sigma_{\beta}}$ GIT semistability agrees with the $\sigma_{\beta}$-semistability defined by inequalities; that is, $\Rep(Q,\alpha)^{SS}_{L_{\sigma_{\beta}}}=\Rep(Q,\alpha)^{\sigma_{\beta}-SS}$ \cite[Proposition 3.1]{KingFDAlgebras}.
\end{itemize} 
Following closely the proof of Pauly \cite[Theorem 3.3]{PaulyEspaces} of the analogous fact for moduli of parabolic bundles, we will now show that the line bundle $L_{\sigma_{\beta}}|_{\Rep(Q,\alpha)^{\sigma_{\beta}-SS}}$ descends to an ample line bundle on $Y_{\alpha,\beta}$. To do this, we recall the descent lemma below due to Kempf \cite[Theorem 2.3]{DrezetNarasimhan}.

\begin{Lemma}\label{KempfDescentLemma}
Let $G$ be a reductive linear algebraic group acting on a $\complex$-variety $X$. Let $f:X\rightarrow Y$ be a good quotient of $X$ by $G$ and $E$ a $G$-equivariant vector bundle on $X$. Then $E$ descends to $Y$ if and only if for each closed point $x\in X$ whose orbit is closed, one has that the stabilizer of $x$ in $G$ acts trivially on the fiber $E|_x$.
\end{Lemma}

Let $V$ be a $\sigma_{\beta}$-semistable representation whose orbit is closed in $\Rep(Q,\alpha)^{\sigma_{\beta}-SS}$. The stabilizer $S_V$ of $V$ in $\text{GL}(Q,\alpha)$ is the group of invertible elements of $\Hom_Q(V,V)$. If $V$ is stable, then we claim $S_V=\complex^*\cdot\text{Id}$, by the following simple argument. If $g:V\rightarrow V$ is an automorphism of $Q$ representations, then choosing some $x$ for which $V(x)$ is nonzero, the isomorphism $g(x):V(x)\rightarrow V(x)$ has a nonzero eigenvalue $\lambda$. Thus, $g-\lambda\cdot\text{Id}$ has a nontrivial kernel. By stability of $V$ and general nonsense for abelian categories with stability structure \cite{RudakovStability}, it follows that $g-\lambda\cdot\text{Id}=0$, whence the claim. The automorphism $\lambda\cdot\text{Id}$ acts on the fiber in $L_{\sigma_{\beta}}$ over $V$ by $\lambda$ to the power of $-\sum_{x\in Q_0}\alpha(x)\sigma_{\beta}(x)=\langle \alpha,\beta \rangle_Q=0$, as desired.

Now consider the general case where $V$ may not be strictly stable. By \cite[Propostion 3.2]{KingFDAlgebras}, we can assume $V$ is a direct sum of $\sigma_{\beta}$-stable representations $V=m_1 V_1\oplus...\oplus m_t V_t$ which satisfy $\sigma_{\beta}(\dim V_i)=0$. Here $V_i$ is not isomorphic to $V_j$ if $i\neq j$. The stabilizer $S_V$ of $V$ is the group of invertible elements of $\Hom_Q(V,V)$, which, arguing as above via \cite{RudakovStability}, is isomorphic to $\text{GL}(m_1)\times...\times\text{GL}(m_t)$. Here we identify $\text{GL}(m_1)\times...\times\text{GL}(m_t)$ with the subgroup of $\text{GL}(Q,\alpha)$ consisting of $g$ such that $g(x)$, taking an appropriate basis for the direct sum, is represented by a block diagonal matrix $\text{diag}(A_1(x),...,A_t(x))$, where $A_i(x)$ is a $m_i\cdot\dim(V_i(x))\times m_i\cdot\dim(V_i(x))$ block matrix, with $m_i^2$-many scalar matrix blocks of size $\dim(V_i(x))\times\dim (V_i(x))$. The scalars that appear in these blocks do not depend on $x\in Q_0$.

Since a $1$-dimensional representation of the general linear group must be given by a power of the determinant, the action of $(g_1,...,g_t)\in S_V$ on the fiber of $L_{\sigma_{\beta}}$ over $V$ is multiplication by $\prod_{i=1}^{t}\det(g_i)^{a_i}$ for some integers $a_i$. For $i=1,..,t$, define $1$-parameter subgroups $h_i$ of $S_V$ by  $h_{i}(\lambda)=(\text{Id},...,\text{Id},\lambda\cdot\text{Id},\text{Id},...,\text{Id})$, where the $\lambda$ appears in the $i$th factor. On the one hand, $h_i$ acts on $L_{\sigma_{\beta}}|_{V}$ by $\lambda^{m_ia_i}$. On the other hand, regarding $S_V$ as a subgroup of $\text{GL}(Q,\alpha)$ as above, we see that $h_i$ acts on $L_{\sigma_{\beta}}|_{V}$ by $\lambda^{y_i}$ where
\begin{equation*}
y_i=-\sum_{x\in Q_0} m_i\cdot\dim(V_i(x))\sigma_{\beta}(x)=-m_i\cdot\sigma_{\beta}(\dim V_i)=0.
\end{equation*}
Therefore, $a_i=0$ for all $i=1,...,t$ and $S_V$ acts trivially on $L_{\sigma_{\beta}}|_V$, as desired.

\end{section}

\begin{section}{Useful Inductive Structure}\label{QuiverH1Prop}
The following proposition allows us to complete the proof of Theorem \ref{QuiverKTT} by an induction argument. It may be of independent interest outside of this proof. For example, it can be used to simplify the existing proof of the quiver-generalized Fulton conjecture \cite{DerksenWeymanCombinatorics}, although we do not do this here. 
\begin{Proposition}\label{QuiverH1PropStatement}
Fix $V\in\Rep(Q,\alpha)$. Let $U_{V}$ be a dense open subset of $\Rep(Q,\beta)$ with the following properties:
\begin{itemize}
\item[i.  ] $\dim\Hom_Q(V,W)$ does not depend on $W\in U_V$.
\item[ii. ] There is a dimension vector $\gamma$ such that for every $W\in U_V$, a dense open subset of $\Hom_Q(V,W)$ consists of morphisms $\phi$ of rank $\gamma$.
\end{itemize}
Now, fix some $W$ in $U_V$. If $\phi\in\Hom_Q(V,W)$ has rank $\gamma$ and $\ker\phi=S\in\Rep(Q,\alpha-\gamma)$, then the canonical surjection $\Ext_Q(V,W)\twoheadrightarrow\Ext_Q(S,W)$ is an isomorphism. 
\end{Proposition}
\begin{proof}
Let
\begin{equation*}
\mathbf{H}=\{(W',\phi')\in U_{V}\times\Hom_Q(V,W'):\phi'\textnormal{ has rank }\gamma\}
\end{equation*}
Note that $\mathbf{H}$ is an open subset of the total space of a vector bundle over $U_V$, hence irreducible of dimension
\begin{equation}\label{QuiverH1Prop1}
\dim\mathbf{H}=\dim(\Rep(Q,\beta))+\dim\Hom_Q(V,W).
\end{equation}
Denote by $\Grass(\alpha-\gamma,V)$ the space of $(\alpha-\gamma)$-dimensional subrepresentations of $V$. Then we have a map $\mathbf{H}\rightarrow\Grass(\alpha-\gamma,V)$ which sends $(W',\phi')$ to $\ker\phi'$. The fiber over a point $S'$ is an open subset of the space of points $(W',\phi')$, where $\phi'\in\Hom_Q(V/S',W')$. 

Define an intermediate space $\mathbf{H}'$ with $\mathbf{H}\rightarrow\mathbf{H}'\rightarrow\Grass(\alpha-\gamma,V)$ such that the fiber in $\mathbf{H'}$ over $S'\in\Grass(\alpha-\gamma,V)$ is given by the open subset of $\prod_{x\in Q_0}\Hom((V/S')(x),\mathbb{C}^{\beta(x)})$ consisting of $\phi'$ such that $\phi'(x)$ is injective for all $x$. Clearly, $\mathbf{H'}\rightarrow\Grass(\alpha-\gamma,V)$ is smooth ($\mathbf{H'}$ is an open subset of a vector bundle over $\Grass(\alpha-\gamma,V)$) and
\begin{equation}\label{QuiverH1Prop2}
\textnormal{reldim}(\mathbf{H'}\rightarrow\Grass(\alpha-\gamma,V))=\sum_{x\in Q_0}\gamma(x)\beta(x).
\end{equation}
Next observe that the fiber in $\mathbf{H}$ over $(S',\phi')\in\mathbf{H'}$ is given by the space of all $W'\in U_V$ such that $\phi'$ is an injective morphism of representations $V/S'\rightarrow W'$. The condition imposed on each arrow $a$ in $W'$ is that $\phi'(ha)\circ(V/S')(a)=W'(a)\circ\phi'(ta)$. Regarding $W'(a)$ as a $\beta(ta)\times\beta(ha)$ matrix with respect to appropriately chosen bases, this equation determines $\gamma(ta)\beta(ha)$ coordinates of $W'(a)$. Thus, we obtain:
\begin{equation}\label{QuiverH1Prop3}
\textnormal{reldim}(\mathbf{H}\rightarrow\mathbf{H}')=\dim(\Rep(Q,\beta))-\sum_{a\in Q_1}\gamma(ta)\beta(ha).
\end{equation} 
Therefore combining (\ref{QuiverH1Prop1}), (\ref{QuiverH1Prop2}), and (\ref{QuiverH1Prop3}) we obtain:
\begin{equation}\label{QuiverH1Prop4}
\dim\Hom_Q(V,W)\leq\dim(\Grass(\alpha-\gamma,V)\textnormal{ at }S)+\langle\gamma,\beta\rangle_{Q},
\end{equation}
where the first summand on the right hand side of (\ref{QuiverH1Prop4}) is the dimension of the largest irreducible component of $\Grass(\alpha-\gamma,V)$ passing through the point $S$. This is at most the dimension of the scheme-theoretic tangent space to $\Grass(\alpha-\gamma,V)$ at $S$, which is $\Hom_Q(S,V/S)$ \cite[Lemma 3.2]{Schofield}. From (\ref{QuiverH1Prop4}), it now follows that
\begin{equation}\label{QuiverH1Prop5}
\dim\Hom_Q(V,W)\leq\dim\Hom_Q(S,V/S)+\langle\gamma,\beta\rangle_{Q}.
\end{equation}
The given map $\phi:V\rightarrow W$ with kernel $S$ induces an injection $\Hom_Q(S,V/S)\hookrightarrow\Hom_Q(S,W)$. It follows that
\begin{equation}\label{QuiverH1Prop6}
\dim\Hom_Q(V,W)\leq\dim\Hom_Q(S,W)+\langle\gamma,\beta\rangle_{Q}.
\end{equation}
Since $\langle\gamma,\beta\rangle=\langle\alpha,\beta\rangle-\langle\alpha-\gamma,\beta\rangle$, the inequality (\ref{QuiverH1Prop6}) can be rewritten as $\dim\Ext_Q(V,W)\leq\dim\Ext_Q(S,W)$. The proof is complete.
\end{proof}
\end{section}

\begin{section}{Outline of the Proof of Theorem \ref{QuiverKTT} by Way of Proposition \ref{QuiverTranslation}}\label{QuiverOutline}
Let $Q$ be a quiver without oriented cycles, $\alpha$, $\beta$ dimension vectors with $\langle \alpha,\beta\rangle_Q=0$. Assume $\dim\SI(Q,\alpha)_{\sigma_{\beta}}=2$. By Proposition \ref{QuiverTranslation}, it suffices for the proof of Theorem \ref{QuiverKTT} to show that $\dim Y_{\alpha,\beta}=1$. This will be done by contradiction in Section \ref{QuiverProof}. If $\dim Y_{\alpha,\beta}\geq 2$, it forces $L_{Y}$ to have a base locus. Take an irreducible component of the inverse image of the base locus in $\Rep(Q,\alpha)^{\sigma_{\beta}-SS}$ and let $Z$ be its closure in $\Rep(Q,\alpha)$. Now for a general point of $(V,W)$ of $Z\times\Rep(Q,\beta)$, we have $\Hom_{Q}(V,W)\neq 0$ (that is, the semi-invariant $\det d^{\square}_W$ vanishes at $V$). This statement is to be contradicted.

Indeed, the assumption $\langle\alpha,\beta\rangle_Q=0$ ensures that $\dim\Hom_Q(V,W)=\dim\Ext_Q(V,W)$, so it suffices for the contradiction to show that $\Ext_Q(V,W)=0$. By Proposition \ref{QuiverH1PropStatement}, this is equivalent to $\Ext_Q(S,W)=0$, where $S$ is the kernel of a general morphism $V\rightarrow W$. The tricky part is to show that $(S,W)$ is generic enough in a closed subset of $\Rep(Q,\dim S)\times\Rep(Q,\beta)$ to apply \ref{QuiverH1PropStatement} again. For this, we need a better understanding of $Z$. We show that $Z$ is actually the image in $\Rep(Q,\alpha)$ of a natural map from a certain irreducible scheme $\mathbf{H}_{*}$, constructed in Section \ref{ConstructionOfDominatingScheme}. The simple description (\ref{DefinitionOfHAlphaBetaDeltaEpsilon}) of $\mathbf{H}_{*}$ allows us to show that indeed $(S,W)$ is generic enough for continued application of \ref{QuiverH1PropStatement}. After applying \ref{QuiverH1PropStatement} enough times, using the semistability of $V$, one finds a subrepresentation $S'$ of $S$ (hence of $V$) such that 
\begin{equation*}
0=\Ext_Q(S',W)\cong\Ext_Q(S,W)\cong\Ext_Q(V,W).
\end{equation*}
This gives our contradiction.

Before proceeding to the detailed proof, we isolate a basic principle from linear algebra which proves very useful in the work to follow. In fact, we've already used it once to get equation (\ref{QuiverH1Prop3}).

\begin{BasicPrinciple*}
Let $V_1$ and $V_2$ be finite dimensional vector spaces. Given two subspaces $i_1:S_1\hookrightarrow V_1$ and $i_2:S_2\hookrightarrow V_2$ and a morphism $\phi:S_1\rightarrow S_2$, the space of linear maps $\psi:V_1\rightarrow V_2$ such that $i_2\circ\phi =\psi\circ i_1$ is a closed nonempty subvariety of $\Hom(V_1,V_2)$ isomorphic to $\mathbb{A}^M$, where $M=\dim V_1\dim V_2-\dim S_1\dim V_2$.  
\end{BasicPrinciple*}

\end{section}

\begin{section}{Construction of $\mathbf{H}_{*}$}\label{ConstructionOfDominatingScheme}
For dimension vectors $\alpha$ and $\delta$, we will say $\delta\leq\alpha$ if for all $x\in Q_0$, $\delta(x)\leq\alpha(x)$. Choose dimension vectors $\alpha$, $\delta$, $\epsilon$ with $\epsilon\leq\delta\leq\alpha$. We will first construct a smooth, irreducible scheme 
\begin{equation}\label{DefinitionOfUAlphaDeltaEpsilon}
\mathbf{U}_{\alpha,\delta,\epsilon}=\{(V,S,S',T):V\in\Rep(Q,\alpha),S,S'\in\Grass(\delta,V),T=S\cap S'\in\Grass(\epsilon,V)\}.
\end{equation}
To begin the construction, recall from \cite[Appendix A]{ShermanKTT} the space $A^r_{f,f,g}$ of triples of subspaces $S$, $S'$, $T=S\cap S'$ of $\complex^r$ with dimensions $f$, $f$, and $g$, respectively. It is shown there that this space is smooth and irreducible. Define
\begin{equation*}
A_1:=\prod_{x\in Q_0}A^{\alpha(x)}_{\delta(x),\delta(x),\epsilon(x)}.
\end{equation*}
We will denote points in $A_1$ by $(S,S',T)$, where $S=(S(x))_{x\in Q_0}$, a collection of $\delta(x)$ dimensional subspaces of $\complex^{\alpha(x)}$ and similarly for $S'$ and $T$. 

For each $x$ in $Q_0$, let $\TT(x)$ be the appropriate rank $\epsilon(x)$ universal subbundle of $\OO_{A_1}\otimes\complex^{\alpha(x)}$. Letting $a_1$,...,$a_{|Q_1|}$ denote the arrows in $Q$, form the total space $A^1_1$ of the bundle $\underline{Hom}(\TT(ta_1),\TT(ha_1))$ over $A_1$. Over $A^1_1$, form the total space $A^2_1$ of the bundle $\underline{Hom}(\TT(ta_2)|_{A^1_1},\TT(ha_2)|_{A^1_1})$. Continue in this fashion until all the arrows are expended. Call the resulting space $A_2$, which is evidently irreducible and smooth over $A_1$. It can be described as follows:
\begin{equation*}
A_2=\{(S,S',T,\{\varphi(a)\}):(S,S',T)\in A_1\textnormal{ and }\{\varphi(a)\}\in\prod_{a\in Q_1}\Hom(T(ta),T(ha))\}.
\end{equation*} 
Now we will attach morphisms to the arrows of $S$ so that $T$ with the arrows $\{\varphi(a)\}$ gives a subrepresentation of $S$ with the attached morphisms. 

To do this, the idea is to apply the Basic Principle of Section \ref{QuiverOutline} at each point of $A_2$, once for each arrow in $Q_1$. More formally, let $\SSS(x)$ be the appropriate rank $\delta(x)$ universal subbundle of $\OO_{A_2}\otimes\complex^{\alpha(x)}$. For each $a\in Q_1$, let $\Phi(a)\in\Hom_{\OO_{A_2}}(\TT(ta),\TT(ha))$ be the universal morphism. The inclusion of bundles $\TT(x)\rightarrow\SSS(x)$ allows us to view $\Phi(a)$ as a section of the total space of $\underline{Hom}(\TT(ta),\SSS(ha))$. Let $A^1_2$ be the inverse image of $\textnormal{Im}\Phi(a_1)$ under the smooth, surjective restriction map of total spaces $\underline{Hom}(\SSS(ta_1),\SSS(ha_1))\rightarrow\underline{Hom}(\TT(ta_1),\SSS(ha_1))$ over $A_2$. Thus, $A^1_2$ is a smooth and surjective over $A_2$ and closed in $\Hom(\SSS(ta_1),\SSS(ha_1))$. Moreover, since the restriction map is smooth with irreducible fibers (each isomorphic to an $\mathbb{A}^M$ as in the Basic Principle), we have that $A^1_2$ is irreducible. Similarly, build $A^2_2$ over $A^1_2$, etc. until all arrows are expended. Repeat the procedure for $S'$ to finally obtain 
\begin{multline*}
A_3=\{(S,S',T,\{\varphi(a)\},\{\psi(a)\},\{\psi'(a)\}):(S,S',T,\{\varphi(a)\})\in A_2,\\ \{\psi(a)\}\in\prod_{a\in Q_1}\Hom(S(ta),S(ha)),\{\psi'(a)\}\in\prod_{a\in Q_1}\Hom(S'(ta),S'(ha)),\\ \textnormal{ and } \psi(a)|_{T(ta)}=\psi'(a)|_{T(ta)}=\varphi(a)\textnormal{ for all }a\in Q_1\}.
\end{multline*}
It is irreducible, surjective, and smooth over $A_2$. 

Finally, we construct $\mathbf{U}_{\alpha,\delta,\epsilon}$ as an irreducible, surjective, and smooth scheme over $A_3$ by a procedure similar to the construction of $A_3$ itself. The idea is to create a scheme $A^1_3$ over $A_3$ whose fiber over a point $(S,S',T,\{\varphi(a)\},\{\psi(a)\},\{\psi'(a)\})$ is the inverse image of $(\psi(a_1),\psi'(a_1))$ under the restriction $\Hom(\complex^{\alpha(ta_1)},\complex^{\alpha(ha_1)})$ to $\Hom(S(ta_1),\complex^{\alpha(ha_1)})\oplus\Hom(S'(ta_1),\complex^{\alpha(ha_1)})$. Because $\psi(a_1),\psi'(a_1)$ restrict to the same morphism on $T(ta_1)$, this fiber is irreducible of dimension independent of the point of $A_3$. Hence $A^1_3$ is irreducible, surjective, and smooth over $A_3$. As above, build an appropriate scheme $A^2_3$ over $A^1_3$, and so on, until the desired $\mathbf{U}_{\alpha,\delta,\epsilon}$ is reached.

Now define
\begin{multline}\label{DefinitionOfHAlphaBetaDeltaEpsilon}
\mathbf{H}_{\alpha,\beta,\delta,\epsilon}=\{(V,W,W',\phi,\phi'):V\in\Rep(Q,\alpha),W,W'\in\Rep(Q,\beta),\\ \phi\in\Hom_Q(V,W),\phi'\in\Hom_Q(V,W'),\\\ker\phi,\ker\phi'\in\Grass(\delta,V),(\ker\phi)\cap(\ker\phi')\in\Grass(\epsilon,V)\}
\end{multline}
This can be constructed as an irreducible, smooth scheme over $\mathbf{U}_{\alpha,\delta,\epsilon}$ as follows. Letting $\VV(x)$, $\SSS(x)$, and $\SSS'(x)$ denote the appropriate universal bundles on $\mathbf{U}_{*}$ for $x\in Q_0$, form the total space of the bundle
\begin{equation*}
\prod_{x\in Q_0}(\underline{Hom}((\VV/\SSS)(x),\complex^{\beta(x)}\otimes\OO)\times\underline{Hom}((\VV/\SSS')(x),\complex^{\beta(x)}\otimes\OO)).
\end{equation*}
A point of this total space over $(V,S,S',T)\in\mathbf{U}_{*}$ is given by finite collections of linear maps $\{\phi(x):V/S(x)\rightarrow\complex^{\beta(x)}\}$ and $\{\phi'(x):V/S'(x)\rightarrow\complex^{\beta(x)}\}$. Let $\mathbf{H}'_{\alpha,\beta,\delta\epsilon}$ denote the open locus of the total space where each of these linear maps is injective. It is clearly irreducible, surjective , and smooth over $\mathbf{U}_{*}$. We build an irreducible $\mathbf{H}_{*}$ smoothly over $\mathbf{H}'_{*}$ by attaching spaces of maps $\complex^{\beta(ta)}\rightarrow\complex^{\beta(ha)}$, so that $\{\phi(x)\}$ and $\{\phi'(x)\}$ become morphisms of representations.

To do this, the idea is again repeat applications of the Basic Principle with, for each arrow $a$, the vectors spaces ``$V_1$," ``$V_2$," ``$S_1$," and ``$S_2$" given by $\complex^{\beta(ta)}$, $\complex^{\beta(ha)}$, $(V/S)(ta)$, and $(V/S)(ha)$ respectively, and ``$\phi$" given by $(V/S)(a)$ (and similarly with $S'$ in place of $S$). The formal argument mirrors the construction of $A_3$ over $A_2$.    

\end{section}

\begin{section}{Proof of Theorem \ref{QuiverKTT}}\label{QuiverProof}
We proceed by contradiction via Proposition \ref{QuiverTranslation}. That is, we suppose
\begin{equation*}
2=\dim\SI(Q,\alpha)_{\sigma_{\beta}},
\end{equation*}
and we assume to the contrary that $\dim Y_{\alpha,\beta}\geq2$. Recall from Proposition \ref{GITForQuivers} the ample line bundle $L_Y$ on $Y$. Let $Z\subseteq\Rep(Q,\alpha)$ be the closure of an irreducible component of the preimage of the base locus of $L_Y$. This base locus is nonempty by the dimension assumption on $Y$. For a general element $(W,W')\in \Rep(Q,\beta)^{\times 2}$, the semi-invariants $\det d^{\square}_{W},\det d^{\square}_{W'}$ form a basis for $\SI(Q,\alpha)_{\sigma_{\beta}}=H^0(Y,L_{Y})$ (see Section \ref{QuiverPrelims} and \cite[Section 2]{DerksenWeymanCombinatorics}). In particular, it follows that for a general element $(V,W)\in Z\times\Rep(Q,\beta)$, one has: 
\begin{equation}\label{QuiverContradiction}
\Hom_Q(V,W)\neq 0,
\end{equation}
i.e. $d^V_W$ is noninjective. Let $\delta<\alpha$ be the dimension vector of the kernel of a general morphism of quiver representations $V\rightarrow W$, equivalently such a morphism has rank $\gamma:=\alpha-\delta$. Also let $\epsilon$ be a dimension vector such that given a general element $(V,W,W')\in Z\times\Rep(Q,\beta)^{\times 2}$ and general pair of quiver morphisms $(\phi,\phi')\in\Hom_Q(V,W)\times\Hom_Q(V,W')$, the intersection of $\ker\phi$ and $\ker\phi'$ has dimension $\epsilon$. 

Constructed in Section \ref{ConstructionOfDominatingScheme}, we have the irreducible smooth scheme $\mathbf{H}_{\alpha,\beta,\delta,\epsilon}$, whose closed points are given by $(V,W,W',\phi,\phi')$ which satisfy:
\begin{multline*}
V\in\Rep(Q,\alpha),\textnormal{ }W,W'\in\Rep(Q,\beta),\textnormal{ } \phi\in\Hom_Q(V,W),\textnormal{ }\phi'\in\Hom_Q(V,W'),\\\ker\phi,\ker\phi'\in\Grass(\delta,V),\textnormal{ }(\ker\phi)\cap(\ker\phi')\in\Grass(\epsilon,V).
\end{multline*}
This scheme controls the base locus $Z$ in the sense of the following proposition, whose proof is virtually identical to \cite[Proposition 5.2]{ShermanKTT}.

\begin{Lemma}\label{QuiverDominantMap}
The morphism
\begin{equation}\label{MapToRepQ}
\textnormal{\textbf{H}}_{\alpha,\beta,\delta,\epsilon}\rightarrow \Rep(Q,\alpha)\times\Rep(Q,\beta)\times\Rep(Q,\beta):(V,W,W',\phi,\phi')\mapsto(V,W,W')
\end{equation}
factors through a dominant map $\textnormal{pr}$ to $Z\times\Rep(Q,\beta)\times\Rep(Q,\beta)$.
\end{Lemma}

Now, there is an irreducible space $Z_{\delta,\epsilon}$ describing all pairs $(S,T)$ consisting of a $\delta$-dimensional representation $S$ of $Q$ with an $\epsilon$-dimensional subrepresentation $T\hookrightarrow S$ (to see this, note that $Z_{\delta,\epsilon}$ is a fiber bundle over $\Rep(Q,\epsilon)$). Fix $(S,T,W)$ a general element of $Z_{\delta,\epsilon}\times\Rep(Q,\beta)$. Since $W$ is general, it follows from Lemma \ref{QuiverDominantMap} that $W$ has a $\gamma$-dimensional subrepsentation $W'$ (which is the image of $V\xrightarrow{\phi} W$ for some $V$ in $Z$), and $W'$ has a $(\delta-\epsilon)$-dimensional subrepresentation $W''$ (which is the image of $S'\xrightarrow{\ker\phi'} V\xrightarrow{\phi} W)$. Let $V_0:=S\oplus W'$ (an $\alpha$-dimensional representation) and let $\phi_0:V_0\rightarrow W$ be the obvious map which has rank $\gamma$ and kernel $S$. Observe that $S':=T\oplus W''$ is a second $\delta$-dimensional subrepresentation of $V_0$ which intersects $S$ in the $\epsilon$-dimensional representation $T$. Therefore, $(V_0,S,S',T)\in\mathbf{U}_{\alpha,\delta,\epsilon}$. By the Basic Principle of Section \ref{QuiverOutline}, one can construct a $\beta$-dimensional representation $W'$ and a morphism $\phi_0':V_0\rightarrow W'$ with kernel $S'$. Thus, $(V_0,W,W',\phi_0,\phi_0')$ is a point of $\mathbf{H}_{\alpha,\beta,\delta,\epsilon}$, where $(S=\ker\phi_0,T=(\ker\phi_0\cap\ker\phi'_0),W)$ is a general element of $Z_{\delta,\epsilon}\times\Rep(Q,\beta)$. Thus, $\mathbf{H}_{\alpha,\beta,\delta,\epsilon}$ dominates $Z_{\delta,\epsilon}\times\Rep(Q,\beta)$, and we have the Lemma below.

\begin{Lemma}\label{QuiverGenericityOfInducedStructures}
Let $(V,W,W',\phi,\phi')$ be a general element of $\mathbf{H}_{\alpha,\beta,\delta,\epsilon}$, with $S:=\ker\phi$, $S':=\ker\phi'$, $T:=S\cap S'$. Then $(S,T,W)$ is a general element of $Z_{\delta,\epsilon}\times\Rep(Q,\beta)$ (e.g. the pair $(S,W)$ is suitable for application of Proposition \ref{QuiverH1PropStatement}).
\end{Lemma}

Now take a general element of $(V,W,W',\phi,\phi')\in\mathbf{H}_{\alpha,\beta,\delta,\epsilon}$ with $S:=\ker\phi$, $S':=\ker\phi'$, $T:=S\cap S'$. The above discussion shows that we may assume:
\begin{itemize}
\item[  i.] $V$ is $\sigma_{\beta}$-semistable.
\item[ ii.] $(V,W)$ is a general element of $Z\times\Rep(Q,\beta)$.
\item[iii.] $(S,T,W)$ is a general element of $Z_{\delta,\epsilon}\times\Rep(Q,\beta)$.
\end{itemize}
By ii and Proposition \ref{QuiverH1PropStatement}, we have $\Ext_Q(V,W)\cong\Ext_Q(S,W)$. By i, every subrepresentation $R$ of $S$ satisfies $\langle\dim R,\beta\rangle\geq 0$. By iii, we may apply Proposition \ref{QuiverFinalInduction} to $(S,T,W)$ and conclude that $\Ext_Q(V,W)=0$. Hence, 
\begin{equation}
\dim\Hom_Q(V,W)=\langle\alpha,\beta\rangle_Q+\dim\Ext_Q(V,W)=0+0
\end{equation}
This contradicts (\ref{QuiverContradiction}).
\end{section}

\begin{section}{Vanishing of Ext for $S$}\label{QuiverDescendToS}
We now ignore the previous context and prove Proposition \ref{QuiverFinalInduction} independent of the foregoing discussion. Let $\epsilon\leq\delta$ be dimension vectors. Let $Z_{\delta,\epsilon}$ be the irreducible space consisting of a $\delta$-dimensional representation $S$ with an $\epsilon$-dimensional subrepresentation $T$. The following is a variant of \cite[Theorem 3]{DerksenWeymanSemiInvariants} (see also \cite[Proposition 6.2]{ShermanKTT}).

\begin{Proposition}\label{QuiverFinalInduction}
Suppose $(S_0,T_0,W_0)\in Z_{\delta,\epsilon}\times\Rep(Q,\beta)$ is a general element. Suppose moreover that every subrepresentation $R$ of $S_0$ satisfies the inequality $\langle\dim R,\beta\rangle\geq0$. Then $\Ext_Q(S_0,W_0)=0$.
\end{Proposition}
\begin{proof}
We proceed by induction on the number $M_{\delta}:=\sum_{x\in Q_0}\delta(x)$. If $M_{\delta}=0$, then the conclusion holds trivially. Assume $M_{\delta}\geq1$. Let $\tilde{\delta}\leq\delta$ be a dimension vector such that if $(S,T,W)$ is a general point of $Z_{\delta,\epsilon}\times\Rep(Q,\beta)$, then a general morphism of representations $\psi:S\rightarrow W$ has kernel of dimension $\tilde{\delta}$. If $\tilde{\delta}=\delta$, then for a general $(S,T,W)$, one has $\Hom_Q(S,W)=0$. On the other hand, by assumption 
\begin{equation*}
0\leq\langle\delta,\beta\rangle=\dim\Hom_Q(S_0,W_0)-\dim\Ext_Q(S_0,W_0),
\end{equation*}
so the conclusion follows in this case. We may as well assume then that $M_{\tilde{\delta}}<M_{\delta}$.

Suppose also that for a general $(S,T,W)$, the $\tilde{\delta}$-dimensional kernel $\tilde{S}$ of a general morphism $\psi:S\rightarrow W$ meets $T$ in an $\tilde{\epsilon}$-dimensional subreprsentation $\tilde{T}$. Let $\mathbf{U}_{\delta,\epsilon,\tilde{\delta},\tilde{\epsilon}}$ be the irreducible smooth scheme whose points are $(S,T,\tilde{S},\tilde{T})$ of the corresponding dimensions such that $S\supseteq T$, $S\supseteq\tilde{S}$, and $\tilde{T}=T\cap\tilde{S}$ (the construction is identical to that of Section \ref{ConstructionOfDominatingScheme}). Build over $\mathbf{U}_{\delta,\epsilon,\tilde{\delta},\tilde{\epsilon}}$ the smooth, irreducible scheme $\mathbf{H}_{\delta,\beta,\epsilon,\tilde{\delta},\tilde{\epsilon}}$ whose fiber over $(S,T,\tilde{S},\tilde{T})$ is the space of $(W,\psi)$ where $W$ is a $\beta$-dimensional representation and $\psi:S\rightarrow W$ has kernel $\tilde{S}$. By choice of $\tilde{\delta}$, $\tilde{\epsilon}$, the map $\pi:\mathbf{H}_{\delta,\beta,\epsilon,\tilde{\delta},\tilde{\epsilon}}\rightarrow Z_{\delta,\epsilon}\times\Rep(Q,\beta)$ is dominant.

Let $(\tilde{S},\tilde{T},W)$ be a general element of $Z_{\tilde{\delta},\tilde{\epsilon}}\times\Rep(Q,\beta)$. Since $W$ is general, it possesses a $(\delta-\tilde{\delta})$-dimensional subrepresentation $W'$, which in turn has a $(\epsilon-\tilde{\epsilon})$-dimensional subrepresentation $W''$ (see the argument preceding Lemma \ref{QuiverGenericityOfInducedStructures}). Now consider $S:=\tilde{S}\oplus W'$ and the obvious morphism $\psi:S\rightarrow W$ with kernel $S$ and rank $\delta-\tilde{\delta}$. If $T:=\tilde{T}\oplus W''$, then $(S,T,\tilde{S},\tilde{T},W,\psi)$ is an element of $\mathbf{H}_{\delta,\beta,\epsilon,\tilde{\delta},\tilde{\epsilon}}$. Since $(\tilde{S},\tilde{T},W)$ is generic, this proves the map $\tilde{\pi}:\mathbf{H}_{\delta,\beta,\epsilon,\tilde{\delta},\tilde{\epsilon}}\rightarrow Z_{\tilde{\delta},\tilde{\epsilon}}\times\Rep(Q,\beta)$ is also dominant. In particular, if $\psi_0$ is a general element of the fiber in $\mathbf{H}_{\delta,\beta,\epsilon,\tilde{\delta},\tilde{\epsilon}}$ over the general element $(S_0,T_0,W_0)$, we can assume that the induced element $(\tilde{S}_0:=\ker\psi_0,\tilde{T}_0:=\tilde{S}_0\cap T_0,W_0)$ of $Z_{\tilde{\delta},\tilde{\epsilon}}\times\Rep(Q,\beta)$ is generic.

Now by Proposition \ref{QuiverH1PropStatement}, $\Ext_Q(S_0,W_0)\cong\Ext_Q(\tilde{S}_0,W)$. Clearly every subrepresentation of $\tilde{S}_0$, being also a subrepresentation of $S_0$, satisfies the appropriate inequality. By genericity of $(\tilde{S}_0,\tilde{T}_0,W_0)$, the inductive hypothesis now completes the proof.
\end{proof}
\end{section}

\begin{section}{Connection to Invariants of Tensor Products}\label{LWConnection}
We now show how Theorem \ref{QuiverKTT} gives the main theorem of \cite{ShermanKTT} as a corollary. Indeed, the relationship between semi-invariants of so-called triple flag quivers and $\SL_r$ invariants of three-fold tensor products is well-known \cite[Section 3 Proposition 1]{DerksenWeymanSemiInvariants}. We prove a geometric generalization, namely that the polarized moduli space of semistable representations is isomorphic to the polarized moduli space of semistable parabolic vector spaces, where in both cases semistability is determined by given Young diagrams $\lambda^1,...,\lambda^s$. 

For $p=1,...,s$ with $s\geq 3$, let $\lambda^p$ be a partition with at most $r-1$ nonzero parts and no part greater than $\ell$. Assume also the partitions satisfy the ``codimension condition:'' 
\begin{equation}\label{ClassicalCodimCond}
r\ell-\sum_{p=1}^s\sum_{a=1}^{r-1}\lambda^p_a=0
\end{equation} 
Note that there must be some such $\ell$ if the tensor product corresponding to $\lambda^1,...,\lambda^s$ is to have invariants.

Let $0<\delta^p_1<...<\delta^p_{C(\lambda^p)}<r$ be the distinct column lengths in $\lambda^p$ and suppose that there are $b(\lambda^p)_i$-many columns of length $\delta^p_i$. Let $X^p$ be the partial flag variety of flags 
\begin{equation*}
0\subset F^p_{\delta^p_1}\subset F^p_{\delta^p_2}\subset...\subset F^p_{\delta^p_{C(\lambda^p)}}\subset \complex^r,
\end{equation*}
where subscripts denote dimensions. We have $\text{Pic}(X^p)=\oplus_{i=1}^{C(\lambda^p)}\mathbb{Z}[L_i]$, where $L_i$ is the pullback of the ample generator of $\text{Pic}(\Grass(\delta^p_i,\complex^r))\cong \mathbb{Z}$ along the canonical projection from $X^p$ \cite{BrionLecturesGeometryOfFlags}. Each $L_i$ has a canonical $\SL_r$-equivariant structure compatible with the usual $\SL_r$ action on $X^p$, so that every line bundle on $X^p$ obtains such a structure. In the sequel, we will take this equivariant structure as implicit. 

With the above description of $\text{Pic}(X^p)$, one has an $\SL_r$-equivariant line bundle
\begin{equation*}
\tilde{\LL}_{\lambda}=\prod_{p=1}^s(b(\lambda^p)_1,...,b(\lambda^p)_{C(\lambda^p)})
\end{equation*}
on $X:=\prod_{p=1}^s X^p$ (for $\SL_r$ acting diagonally). The semistable points with respect to this linearization are those $\FF=\{F^p_{\bullet}\}_{p=1}^s\in X$ such that if $S$ is an $r'$ dimensional subspace of $\complex^r$, then 
\begin{equation}\label{SemistabilityForPartialFlags}
\sum_{p=1}^s\sum_{i=1}^{C(\lambda^p)} b(\lambda^p)_i\dim(F^p_{\delta^p_i}\cap S)\leq r'\ell.
\end{equation}
There is an integral, projective good quotient $\rho:X^{SS}\rightarrow \MM_{\lambda}$ for the action of $\SL_r$, where $\MM_{\lambda}$ has rational singularities. The line bundle $\tilde{\LL}_{\lambda}$ descends to an ample line bundle $\LL_{\lambda}$ on $\MM_{\lambda}$. Moreover, one has a natural isomorphism for each positive integer $n$:
\begin{equation*}
H^0(\MM_{\lambda},\LL_{\lambda}^{\otimes n})=(V^*_{n\lambda^1}\otimes...\otimes V^*_{n\lambda^s})^{\SL_r}.
\end{equation*}
See \cite[Section 2]{ShermanKTT} for a summary with appropriate references. 

We saw similarly in Section \ref{QuiverGIT} that for a cycle-free quiver $Q$ and dimension vectors $\alpha,\beta$ of $Q$, one has a moduli space with an ample line bundle $(Y_{\alpha,\beta},L_{Y})$, where sections of tensor powers of $L_Y$ give $\sigma_{\beta}$ semi-invariants of $\Rep(Q,\alpha)$. The goal now is to show that for the right of choice of $Q,\alpha,\beta$, the polarized moduli spaces $(Y_{\alpha,\beta},L_Y)$ and $(\MM_{\lambda},\LL_{\lambda})$ are actually the same. To this end, let $Q$ be the $s$-partial flag quiver of vertices labeled $1^p, 2^p, ..., C(\lambda^p)^p$ for $p=1,...,s$ and one additional vertex $r=(C(\lambda^1)+1)^1=...=(C(\lambda^s)+1)^s$, with arrows $i^p\rightarrow (i+1)^p$ for all $p=1,...,s$ and $i=1,...,C(\lambda^p)$. Let $\alpha$ be the dimension vector given by $\alpha(i^p)=\delta^p_i$ and $\alpha(r)=r$. For example, if $r=4$, $\ell=5$, $s=3$, $\lambda^1=5\geq 2\geq 1$, $\lambda^2=\lambda^3=4\geq 2$, an element of $\Rep(Q,\alpha)$ is depicted below.
\begin{equation*}
\begindc{\commdiag}[50]
\obj(0,2)[02]{$\complex$}
\obj(1,2)[12]{$\complex^2$}
\obj(2,2)[22]{$\complex^3$}
\mor{02}{12}{$\phi^1_1$}
\mor{12}{22}{$\phi^1_2$}
\obj(1,1)[11]{$\complex$}
\obj(2,1)[21]{$\complex^2$}
\mor{11}{21}{$\phi^2_1$}
\obj(1,0)[10]{$\complex$}
\obj(2,0)[20]{$\complex^2$}
\mor{10}{20}{$\phi^3_1$}
\obj(3,1)[Final]{$\complex^4$}
\mor{22}{Final}{$\phi^1_3$}
\mor{21}{Final}{$\phi^2_2$}
\mor{20}{Final}{$\phi^3_2$}
\enddc
\end{equation*}
Let $\beta$ be the dimension vector $\beta(i^p)=\ell-\sum_{i=1}^{C(\lambda^p)}b(\lambda^p)_i$. Then $\langle \alpha,\beta\rangle_Q=0$ because of (\ref{ClassicalCodimCond}). As in Section \ref{QuiverPrelims}, we have an associated weight $\sigma_{\beta}$, given in this case by $\sigma_{\beta}(i^p)=b(\lambda^p)_i$ and $\sigma_{\beta}(r)=-\ell$. 

Let $G=\prod_{p=1}^s\prod_{i=1}^{C(\lambda^p)}\text{Aut}(\complex^{\delta^p_i})$, so that $\text{GL}(Q,\alpha)=G\times\text{Aut}(\complex^r)$. Denote by $\Rep(Q,\alpha)_{\text{inj}}$ the open locus of $\phi\in\Rep(Q,\alpha)$ such that $\phi^p_i$ is injective for all $p=1,...,s$ and $i=1,...,C(\lambda^p)$. One has a map $f:\Rep(Q,\alpha)_{\text{inj}}\rightarrow\prod_{p=1}^s X^p$ which sends $\phi$ to the $s$-tuple of flags whose $p$th flag $F^p_{\bullet}$ is given by $F^p_{\delta^p_i}=\text{Im}(\phi^p_{C(\lambda^p)}\circ...\circ\phi^p_{i+1}\circ\phi^p_i)$.

\begin{Lemma}\label{FlagsAreAGITQuotientOfQuiverReps}
The map $f$ is a geometric quotient of $\Rep(Q,\alpha)_{\textnormal{inj}}$ by $G\times\{\textnormal{Id}\}$.
\end{Lemma} 
\begin{proof}
The case $s=1$, $X=\Grass(a,\complex^r)$ is standard (see e.g. \cite[Section 8.1]{MukaiInvariants}). The general case proceeds much the same way. Linearize the action of $G\times\{\text{Id}\}$ on $\Rep(Q,\alpha)$ by the character which takes $(g_x)_{x\in Q_0-\{r\}}$ to $\prod_{x\in Q_0-\{r\}}\det(g_x)$. It is easy to see that with respect to this linearization, we have $\Rep(Q,\alpha)^{SS}=\Rep(Q,\alpha)^S=\Rep(Q,\alpha)_{\text{inj}}$. Therefore, one has a geometric quotient $Y$ of $\Rep(Q,\alpha)_{\text{inj}}$. Clearly $f$ is constant on $G\times\{\text{Id}\}$ orbits, so $f$ descends to $\tilde{f}$ on the geometric quotient of $\Rep(Q,\alpha)_{\text{inj}}$. Now $\tilde{f}$ is surjective (because $f$ is), and $\tilde{f}$ is injective by the following simple argument. If $\phi=\prod_{p=1}^s(\phi^p_1,...,\phi^p_{r-1})$ and $\phi'=\prod_{p=1}^s(\phi'^p_1,...,\phi'^p_{r-1})$ have the same image in $X$, then in particular, $\phi^p_{r-1}$ and $\phi'^p_{r-1}$ have the same image in $\complex^r$. Thus, there is a $g^p_{r-1}\in\text{GL}_{r-1}$ such that $\phi'^p_{r-1}=\phi^p_{r-1}\circ (g^p_{r-1})^{-1}$. Similarly, we have $\psi^p_{r-2}:=\phi^p_{r-1}\circ\phi^p_{r-2}$ equals $\psi'^p_{r-2}$ defined likewise, whence there is $g^p_{r-2}\in\text{GL}_{r-2}$ such that $\psi'^p_{r-2}=\psi^p_{r-2}\circ (g^p_{r-2})^{-1}$. Expanding this out and canceling $\phi^p_{r-1}$ on opposite sides, we obtain $\phi'^p_{r-2}=g^p_{r-1}\circ\phi^p_{r-2}\circ(g^p_{r-2})^{-1}$. Continuing in this fashion, we see that $\phi$ and $\phi'$ are in the same orbit. By Zariski's main theorem, $\tilde{f}$ is an isomorphism.
\end{proof}

Now consider a $\sigma_{\beta}$-semistable point $\phi\in\Rep(Q,\alpha)^{\sigma_{\beta}-SS}$. If for some $1\leq i_0\leq r-1$, $1\leq p_0\leq s$, the component $\phi^{p_0}_{i_0}$ has a kernel containing a nonzero vector $v\in\complex^{\delta^{p_0}_{i_0}}$, then $\phi$ has a subrepresentation of dimension vector $(\dim\psi)(i^p)=0$ for $i^p\neq i^{p_0}_0$ and $(\dim\psi)(i^{p_0}_0)=1$. It is given by the map $\complex\rightarrow\complex^{\delta^{p_0}_{i_0}}$ which sends $1$ to $v$. In this case we have $\sigma_{\beta}(\dim\psi)=b(\lambda^{p_0})_{i_0}$, which is strictly positive, violating semistability of $\phi$. Thus, $\Rep(Q,\alpha)^{\sigma_{\beta}-SS}\subseteq\Rep(Q,\alpha)_{\text{inj}}$. Now an easy calculation comparing the inequalities (\ref{SemistabilityForPartialFlags}) to those of $\sigma_{\beta}$-semistability shows that $f^{-1}(X^{SS})=\Rep(Q,\alpha)^{\sigma_{\beta}-SS}$. Therefore, by Lemma \ref{FlagsAreAGITQuotientOfQuiverReps}, the variety $X^{SS}$ is the good quotient $\Rep(Q,\alpha)^{\sigma_{\beta}-SS}//(G\times\{\text{Id}\})$. Since a morphism from $\Rep(Q,\alpha)$ is $\text{GL}(Q,\alpha)$ invariant if and only if it is $G\times\SL_r$ invariant, one has:
\begin{equation*}
Y_{\alpha,\beta}=\Rep(Q,\alpha)^{\sigma_{\beta}-SS}//(G\times\SL_r)=(X^{SS})//\textnormal{SL}_r=\MM_{\lambda}.
\end{equation*} 

It now remains to prove that the line $\LL_{\lambda}$ and $L_Y$ agree under this identification. We will need some lemmas.

\begin{Lemma}\label{LineBundlesAgree1}
A section $s$ in $H^0(\Rep(Q,\alpha),L_{\sigma_{\beta}})$ is $\textnormal{GL}(Q,\alpha)$ invariant if and only if it is $G\times\SL_r$ invariant.
\end{Lemma}
\begin{proof}
Necessity is immediate. For sufficiency, suppose $\bar{g}\in\text{GL}(Q,\alpha)$. We may write this element as $\bar{g}=\left(\times_{p=1}^{s}\times_{i=1}^{C(\lambda^p)}\bar{g}^p_i\right)\times \bar{g}_r$. Let $t$ be an $r$th root of $\det\bar{g}_r$ and write
\begin{equation*}
\bar{g}^1=\left(\times_{p=1}^{s}\times_{i=1}^{C(\lambda^p)}t\cdot\text{Id}_{\complex^{\delta^p_i}}\right)\times\left(t\cdot\text{Id}_{\complex^{r}}\right)\textnormal{ and }\bar{g}^2=\left(\times_{p=1}^{s}\times_{i=1}^{C(\lambda^p)}\frac{\bar{g}^p_i}{t}\right)\times\frac{\bar{g}_r}{t},
\end{equation*}
so that $\bar{g}=\bar{g}^1\cdot\bar{g}^2$. Since $\bar{g}^2\in G\times\SL_r$ and $\bar{g}^1$ acts trivially on sections of $L_{\sigma_{\beta}}$ (by (\ref{ClassicalCodimCond})), if $s$ is $G\times\SL_r$ invariant, we have $\bar{g}\cdot s=s$, as desired.
\end{proof}

\begin{Lemma}\label{LineBundlesAgree2}
Let $0<\delta_1<...<\delta_C<r$ be integers, let
\begin{equation*}
H=\Hom(\complex^{\delta_1},\complex^{\delta_2})\times...\times\Hom(\complex^{\delta_{C-1}},\complex^{\delta_{C}})\times\Hom(\complex^{\delta_C},\complex^r),
\end{equation*}
and $H_{\textnormal{inj}}$ the locus where $\phi_1,...,\phi_C$ are all injective. Observe that $H_{\textnormal{inj}}$ has a natural conjugation action of $\GG\times\textnormal{SL}_r$ where $\GG:=\prod_{i=1}^{C}\textnormal{Aut}(\complex^{\delta_i})$. For $i=1,...,C$, let $\GG$ act trivially on $\Grass(\delta_i,\complex^r)$ and its ample generator $\OO(1)$, and let $\SL_r$ act in the usual way on both of these. If $f_i:H_{\textnormal{inj}}\rightarrow\Grass(\delta_i,\complex^r)$ is the $\GG\times\SL_r$-equivariant map sending $(\phi_1,...,\phi_C)$ to $\textnormal{Im}(\phi_C\circ...\circ\phi_{i})$, then the pullback along $f_i$ of $\OO(1)$ is $\GG\times\SL_r$ equivariantly isomorphic to $L_i$ on $H_{\textnormal{inj}}$. Here we denote by $L_i$ the $\GG\times\SL_r$ equivariant bundle whose underlying bundle is trivial and whose action is given by $g\cdot(\phi,z)=(g\cdot\phi,(\det{g_i})^{-1}z)$.
\end{Lemma}
\begin{proof}
Let $\SSS_i$ be the universal subbundle on $\Grass(\delta_i,\complex^r)$, endowed with an equivariant structure by allowing $\GG$ to act trivially and $\SL_r$ to act in the obvious way. Endow also the trivial rank $i$ bundle $H_{\text{inj}}\times\complex^i$ with the action $g\cdot(\phi,v)=(g\cdot\phi,g_i v)$. One has a $\GG\times\SL_r$ equivariant isomorphism $\rho:H_{\text{inj}}\times\complex^{i}\rightarrow f^*\SSS_i$ which sends $(\phi,v)$ to $(\phi_{C}\circ...\circ\phi_i)(v)$ in $\SSS|_{f_i(\phi)}$. Thus, $\det\rho$ is a $\GG\times\SL_r$ equivariant isomorphism of $H_{\text{inj}}\times\complex$ (action given by $g\cdot(\phi,z)=(g\cdot\phi,(\det g_i)z)$) with $f_i^*\OO(-1)$. The assertion follows.
\end{proof}

From Lemma \ref{LineBundlesAgree2}, one deduces $f^*\tilde{\LL}^{\otimes n}_{\lambda}$ is $G\times\SL_r$ equivariantly isomorphic to $L^{\otimes n}_{\sigma_{\beta}}$. Thus, using \ref{LineBundlesAgree1} in the first step below, we have:
\begin{multline}\label{LineBundlesAgree3}
H^0(Y_{\alpha,\beta},L^{\otimes n}_Y)=H^0(\Rep(Q,\alpha)^{\sigma_{\beta}-SS},L^{\otimes n}_{\sigma_{\beta}})^{G\times \SL_r}\\=H^0(X^{SS},\tilde{\LL}^{\otimes n}_{\lambda})^{\SL_r}=H^0(\MM_{\lambda},\LL^{\otimes n}_{\lambda}).
\end{multline}
It follows that $L_Y=\LL_{\lambda}$. 

\begin{Corollary}\label{QuiverInvariantsGiveTensorInvariants}
\cite{ShermanKTT} If $\dim(V^*_{\lambda^1}\otimes...\otimes V^*_{\lambda^s})^{\SL_r}=2$, then $\dim(V^*_{n\lambda^1}\otimes...\otimes V^*_{n\lambda^s})^{\SL_r}=n+1$ for all positive integers $n$.
\end{Corollary}
\begin{proof}
The left hand side of (\ref{LineBundlesAgree3}) is $\SI(Q,\alpha)_{\sigma_{n\beta}}$ by Proposition \ref{GITForQuivers} while the right hand side is $(V^*_{n\lambda^1}\otimes...\otimes V^*_{n\lambda^s})^{\SL_r}$. The corollary now follows from Theorem \ref{QuiverKTT}.
\end{proof}
\end{section}

\bibliographystyle{amsalpha}
\bibliography{C:/Users/dothetimmywalk/Desktop/Grad/QuiverIkenmeyerNice/QuiverIkenmeyer}{}
\end{document}